\documentclass[reqno,12pt]{amsart}
\usepackage{amsmath,amsthm,amssymb,amsfonts,amscd,graphicx}

\theoremstyle{plain} 
	\newtheorem{thm}{Theorem}[section]
	\newtheorem*{thm*}{Theorem}

	\newtheorem{prop}[thm]{Proposition}
	\newtheorem{conj}[thm]{Conjecture}
	\newtheorem*{conj*}{Conjecture}
	\newtheorem{assum}[thm]{Assumption}
\theoremstyle{definition}
	\newtheorem{defn}[thm]{Definition}

\theoremstyle{remark}

\input xy
\xyoption{all}

\def\AA{{\mathbb A}}

\def\LL{{\mathbb L}}

\def\RR{{\mathbb R}}

\def\ZZ{{\mathbb Z}}
\def\A{{\mathcal A}}

\def\G{{\mathcal G}}
\def\H{{\mathcal H}}

\def\J{{\mathcal J}}

\def\L{{\mathcal L}}

\def\O{{\mathcal O}}

\def\T{{\mathcal T}}

\def\f{{\mathfrak f}}

\def\m{{\mathfrak m}}

\def\ff{{\mathfrak F}}
\def\p{\partial }

\def\ns{{\nabla}\hspace{-1.4mm}\raisebox{0.3mm}{\text{\footnotesize{\bf /}}}}

\title[]{From Calabi--Yau dg Categories to Frobenius manifolds via Primitive Forms}

\author[]{Atsushi Takahashi}

\address{Department of Mathematics, Graduate School of Science, Osaka University, 
Toyonaka Osaka, 560-0043, Japan}

\email{takahashi@math.sci.osaka-u.ac.jp}

\subjclass[2010]{}


\begin{document}

\begin{abstract}
It is one of the most important problems in mirror symmetry 
to obtain functorially Frobenius manifolds from smooth compact Calabi-Yau $A_\infty$-categories. 
This paper gives an approach to this problem based on the theory of primitive forms. 
Under an assumption on the formality of a certain homotopy algebra, a formal primitive form for a smooth compact 
Calabi--Yau dg algebra can be constructed, which enable us to have a formal Frobenius manifold.
\end{abstract}

\maketitle

\section{Introduction}

Mirror symmetry gives an identification between two objects coming from different mathematical origins. 
It has been studied intensively by many mathematician for more than twenty years since it yields important, interesting and unexpected geometric information.
Almost ten years before the discovery of mirror symmetry, 
K.~Saito studied a deformation theory of an isolated hypersurface singularity in order to generalize 
the theory of elliptic integrals \cite{s:1}.
There he developed a certain generalization of the complex Hodge theory of Calabi--Yau type and 
found a differential geometric structure, which he called a flat structure,  
on the base space of the deformation (cf. \cite{st:1} for a summary). This is known as Saito's theory of primitive forms.
His flat structure was also found later by Dubrovin in his study of 
two dimensional topological field theories in their relation with integrable systems \cite{d:1}, 
which is axiomatized under the name of a Frobenius manifold (cf. \cite{he:1,ma:1,sab:1}).
The classical mirror symmetry conjecture states the existence of an isomorpshism between the Frobenius manifold from 
the Gromov--Witten theory of a Calabi-Yau manifold and the one from the deformation theory of another Calabi-Yau manifold.
In order to explain this mysterious correspondence, 
Kontsevich conjectured in \cite{K} that the derived category of coherent sheaves on a Calabi-Yau manifold should be equivalent to the derived Fukaya category for another Calabi-Yau manifold, 
which is known as the homological mirror symmetry conjecture.
In particular, he expects that the classical mirror symmetry isomorphism can naturally be induced by the homological mirror symmetry equivalence via the moduli space of $A_\infty$-deformations.   
As is already explained in \cite{st:1}, although Saito's construction of the Frobenius manifolds based on 
filtered de Rham cohomology groups, Gau\ss--Manin connctions, higher residue pairings and primitive forms is formulated only 
for isolated hypersurface singularities, several notion and the idea of this construction is of general nature. 
It has inspired and motivated the non-commutative Hodge theory for Calabi--Yau $A_\infty$ categories (c.f. \cite{kakop:1, ks:1}), 
which is the basic ingredient in constructing primitive forms in this categorical language.
The present paper tries to make the contents of Appendix in \cite{st:1}, where we give a list how objects in 
the original Saito's theory of primitive forms may be generalized and rewritten, 
as precise as possible at present for smooth compact Calabi--Yau dg algebras. 
More precisely, in Section~2 we recall some notations and terminologies of differential graded algebras, and then 
we introduce the dg analogue of the weighted homogeneous polynomial with an isolated singularity at the origin.
It is Proposition~\ref{prop:iL}, the ``Cartan calculus", that is the key in our story,
which enable us to translate the original Saito's theory for isolated weighted homogeneous hypersurface singularities 
verbatim into the dg categorical Saito theory. 
In Section~3, we show that a very good section 
can be constructed under a certain formality assumption which is motivated by the formality considered in \cite{BK:1}, 
the identification of the (formal) moduli space of $A_\infty$-deformation of the derived category of coherent sheaves with 
the formal neighborhood of zero in the total cohomology group of polyvectors.
After studying the versal deformation of a Saito structure in Section~4, 
we show how a primitive form is deduced from a very good section based on the famous method developed by M.~Saito~\cite{MSaito} 
and Barannikov~\cite{bar:1,bar:2}. 
\smallskip
\noindent
{\it Acknowledgement}\\
\indent
The author would like to thank organizers of the workshop ``Primitive forms and related subjects".
He is supported by JSPS KAKENHI Grant Numbers 23654009 and 26610008. 
\section{Notations and Terminologies}
In this paper, we denote by $k$ an algebraically closed field of characteristic zero with the unit $1_k$.
\subsection{Differential graded algebras}
On dg algebras and derived categories of modules over them, 
we refer the reader to \cite{Keller}, of which we follow the terminologies and notations. 
A {\em differential graded $k$-algebra} (or simply, {\em dg $k$-algebra}) is 
a $\ZZ$-graded $k$-algebra $A=\bigoplus_{p\in\ZZ}A^p$ 
equipped with a differential, a $k$-linear map $d_A: A\longrightarrow A$ of degree one with $d_A^2=0$ 
satisfying the Leibniz rule$:$
\begin{equation}\label{eq:L}
d_A(a_1  a_2)=(d_Aa_1)a_2+(-1)^{p}a_1(d_Aa_2),\quad a_1\in A^{p}, a_2\in A.
\end{equation}
For a homogeneous element $a\in A^p$, let $\overline{a}$ be the degree of $a$, namely, $\overline{a}:=p$.
Denote by $H^\bullet(A,d_A):=\bigoplus_{p\in\ZZ}H^p(A,d_A)$ the cohomology of $A$ which is a graded $k$-module.
Throughout this paper, $A$ denotes a dg $k$-algebra.
We recall some terminologies for our later use.
For a dg $k$-algebra $A$, $A^e:=A^{op}\otimes_k A$ denotes the enveloping dg $k$-algebra of $A$ where 
$A^{op}$ is the opposite dg $k$-algebra of $A$. 
The dg $A$-module is called perfect if it belongs to the smallest full triangulated subcategory of the derived category of 
dg $A$-modules which contains $A$ and is closed under direct summands and isomorphisms. 
\begin{defn}
Let $A$ be a dg $k$-algebra. 
\begin{enumerate}
\item 
A dg $k$-algebra $A$ is called non-negatively graded if $A^p=0$ for all negative integers $p$.
\item 
A dg $k$-algebra $A$ is called compact if $A$ is a perfect dg $k$-module, 
namely, if its cohomology $H^\bullet(A,d_A)$ is finite dimensional
(cf. Kontsevich--Soibelman, Definition~8.2.1 in \cite{ks:1}). 
\item
A dg algebra $A$ is called smooth if $A$ is a perfect dg $A^e$-module  (cf. Kontsevich--Soibelman, Definition~8.1.2 in \cite{ks:1}).
\item
A non-negatively graded dg $k$-algebra $A$ is called connected if $H^0(A,d_A)=k[1_A]$.
\end{enumerate}
\end{defn}
For dg $A$-modules $M, N$, we shall denote by $\RR{\rm Hom}_A(M,N)$ the $\RR{\rm Hom}$-complex.
Let $A^!:=\RR{\rm Hom}_{(A^e)^{op}}(A,A^e)$ be the inverse dualizing complex.
\begin{prop}
Let $A$ be a smooth dg $k$-algebra. 
The natural morphism in the derived category of dg $A^e$-modules
\[
A\longrightarrow \RR{\rm Hom}_{A^e}(\RR{\rm Hom}_{(A^e)^{op}}(A,A^e),A^e)
=\RR{\rm Hom}_{A^e}(A^!,A^e)
\]
is an isomorphism. 
In particular, we have a natural isomorphism in the derived category of dg $k$-modules
\begin{equation}\label{eq:1}
A\otimes^\LL_{A^e} A\cong A\otimes^\LL_{A^e}\RR{\rm Hom}_{(A^e)^{op}}(A^!,A^e)\cong \RR{\rm Hom}_{A^e}(A^!,A).
\end{equation}
\end{prop}
\subsection{Calabi--Yau dg algebras}
\begin{defn}[Ginzburg, Definition~3.2.3 in \cite{Ginzburg}]
Fix an integer $w\in \ZZ$. A smooth dg $k$-algebra $A$ is called Calabi--Yau of dimension $w$
if there exists an isomorphism 
\begin{equation}\label{defn:CY}
A^!\cong T^{-w}A
\end{equation}
in the derived category of dg $A^e$-modules.
\end{defn}
Let $A$ be a smooth Calabi--Yau dg $k$-algebra of dimension $w$. 
It is important to note that a choice of an element in $H^{0}(\RR{\rm Hom}(A^!,T^{-w}A))$ giving
the isomorphism~\eqref{defn:CY} yields an isomorphism
\begin{equation}
\RR{\rm Hom}_{A^e}(A,A)\cong \RR{\rm Hom}_{A^e}(T^{w} A^!,A)\cong T^{-w}(A\otimes^\LL_{A^e}A),
\end{equation}
in the derived category of dg $k$-modules due to the isomorphism~\eqref{eq:1}.
Recall that for a smooth compact dg $k$-algebra $A$ the Serre functor on the perfect derived category of dg $A^e$-modules
is given by the dg $A^e$-module $A^*:=\RR{\rm Hom}_{k}(A,k)$.
It is known that the dg $A^e$-module $A^!$ gives its quasi-inverse, which implies the following.
\begin{prop}
Let $A$ be a smooth compact Calabi--Yau dg $k$-algebra of dimension $w$. 
There exists an isomorphism in the perfect derived category of dg $A^e$-modules depending on
a choice of an element in $H^{-w}(\RR{\rm Hom}(A^!,A))$:
\begin{equation}
T^{w}A\cong  A^*.
\end{equation}
\end{prop}
\subsection{Hochschild cohomology}
For graded $k$-modules $M$ and $N$, 
$\G r_k (M,N)=\bigoplus_{p\in\ZZ} \G r_k (M,N)^p$ denotes a graded $k$-module 
where $\G r_k (M,N)^p$ is a $k$-module consisting of graded $k$-linear maps from $M$ to $N$ of degree $p$. 
For a dg $k$-algebra $A$, set
\begin{equation}
C^\bullet(A):=\prod_{n\ge 0} \G r_k ((TA)^{\otimes n},A),
\end{equation}
For $a\in A$, denote it by $Ta$ when we consider it as an element of $TA$.
Define a $k$-linear map $d:C^\bullet(A)\longrightarrow C^{\bullet+1}(A)$ by 
$(df)(1_k ):=d_A(f(1_k ))$ for $f\in \G r_k (k,A)$ and, for $f\in \G r_k ((TA)^{\otimes n},A)^p$, $n\ge 1$,
\begin{multline*}
(df)(Ta_1\otimes \dots \otimes Ta_n )
:=d_A\left(f(Ta_1\otimes \dots \otimes Ta_n )\right)\\
-\sum_{i=1}^n(-1)^{p-1+s_{i-1}} f(Ta_1\otimes \dots \otimes Ta_{i-1}\otimes T(d_Aa_i)\otimes Ta_{i+1}\otimes \dots \otimes Ta_n),
\end{multline*}
where $s_i:=\sum_{m=1}^{i}(\overline{a_m }-1)$.
Define another $k$-linear map $\delta:C^\bullet(A)\longrightarrow C^{\bullet+1}(A)$ by 
$(\delta f)(Ta_1):=(-1)^{(p-1)(\overline{a_1}-1)+\overline{a_1}} a_1f(1_k )+(-1)^{p} f(1_k )a_{1}$ for $f\in \G r_k (k,A)^p$ 
and, for $f\in \G r_k ((TA)^{\otimes n},A)^p$, $n\ge 1$, 
\begin{multline*}
(\delta f)(Ta_1\otimes \dots\otimes Ta_{n+1} )
:=(-1)^{p+s_n} f(Ta_1\otimes \dots \otimes Ta_{n})a_{n+1}\\
+\sum_{i=1}^{n}(-1)^{p-1+s_i} f(Ta_1\otimes \dots \otimes Ta_{i-1}\otimes T(a_ia_{i+1})\otimes Ta_{i+2}\otimes \dots \otimes Ta_{n+1})\\
+(-1)^{(p-1)(\overline{a_1}-1)+\overline{a_1}}  a_{1} f(Ta_2\otimes \dots \otimes Ta_{n+1}).
\end{multline*}
The following proposition is well-known.
\begin{prop}
The $k$-linear maps $d, \delta$ define differentials on $C^\bullet(A)$ and 
satisfy $d\delta+\delta d=0$. Moreover, the dg $k$-module $(C^\bullet(A),\p:=d+\delta)$ is isomorphic to 
the dg $k$-module $\RR{\rm Hom}_{A^e}(A,A)$ in the derived category of dg $k$-modules.
\end{prop}
\begin{defn}
The dg $k$-module $(C^\bullet(A),\partial)$ is called the 
Hochschild cochain complex of $A$, whose cohomology $H^\bullet(C^\bullet(A),\partial)$ 
is denoted by $H\!H^\bullet (A)$ and is called the Hochschild cohomology of $A$.
Denote by $\T^\bullet_{poly}(A)$ the graded $k$-module $H^\bullet(C^\bullet(A),\delta)$.
\end{defn}
For $f=(f_n)_{n\ge 0}\in C^p(A)$ and $g= (g_n)_{n\ge 0}\in C^q(A)$, one can define the product 
$f\circ g =\left((f\circ g)_{n}\right)_{n\ge 0}\in C^{p+q}(A)$ by 
\begin{multline*}
(f\circ  g)_n(Ta_1\otimes \dots\otimes Ta_{n} )\\
:=
\sum_{i=0}^n
(-1)^{q\cdot s_i}
f_i(Ta_1\otimes \dots \otimes Ta_{i})g_{n-i}(Ta_{i+1}\otimes \dots \otimes Ta_{n}).
\end{multline*}
It is also known that for $f\in C^p(A)$ and $g\in C^q(A)$ one can define 
the {\it Gerstenhaber bracket} $[f,g]_G\in C^{p+q-1}(A)$ by
\begin{equation}
[f,g]_G:=f\circ_{-1} g-(-1)^{(p-1)(q-1)}g\circ_{-1}f,
\end{equation}
where $f\circ_{-1}g =\left((f\circ_{-1}g)_{n}\right)_{n\ge 0}\in C^{p+q-1}(A)$ is given by 
\begin{multline*}
(f\circ_{-1}g)_n(Ta_1\otimes\dots\otimes Ta_{n})\\
 := \sum_{i=1}^n\sum_{j=1}^{n}(-1)^{(q-1)s_{i-1} }f_{n-j+1}(Ta_1\otimes \dots \otimes Ta_{i-1}\\
\otimes T\left(g_j(Ta_{i}\otimes \dots\otimes Ta_{i+j-1})\right)\otimes Ta_{i+j}\otimes \dots \otimes Ta_n).
\end{multline*}
We collect some basic properties of $\T^\bullet_{poly}(A)$ and $H\!H^\bullet(A)$, which are well-known facts or 
follow from a straight forward calculation.
\begin{prop}
The product $\circ $ on $C^\bullet(A)$ induces structures of graded commutative $k$-algebras on 
$H\!H^\bullet(A)$ and $\T^\bullet_{poly}(A)$ whose unit elements are given by the cohomology classes of $1_A\in C^0(A)$. 
The Gerstenharber bracket $[-,-]_G$ induces structures of graded Lie algebras on 
$HH^{\bullet+1}(A)$ and $\T^{\bullet+1}_{poly}(A)$. 
\end{prop}
Moreover, it turns out that the tuple $(\T^\bullet_{poly}(A),d,\circ ,[-,-]_G)$ is a differential Gerstenhaber algebra (cf. \cite{ger:1}).
Namely, we also have
\[
\left[X,Y\circ  Z\right]_G=\left[X,Y\right]\circ Z
+(-1)^{(\overline{X}+1)\cdot \overline{Y}}Y\circ \left[X,Z\right]_G,\quad X,Y,Z\in \T^\bullet_{poly}(A).
\]
Let $\m_A=(\m_A^p)_{p\ge 0}$ be the element in $C^2(A)$ where 
$\m_A^p:=0$ if $p\ne 1,2$ and $\m_A^1(Ta_1):=d_Aa_1$, $\m_A^2(Ta_1\otimes Ta_2):=(-1)^{\overline{a_1}}a_1a_2$.
It is well-known that $\left[\m_A,\m_A\right]_G=0$ since, if we write $\m_A=\m^1_A+\m^2_A$, 
\[
\left[\m_A,\m_A\right]_G=0\Leftrightarrow
\begin{cases}
\left[\m^1_A,\m^1_A\right]_G=0\\
\left[\m^1_A,\m^2_A\right]_G=0\\
\left[\m^2_A,\m^2_A\right]_G=0
\end{cases} 
\Leftrightarrow
\begin{cases}
d_A^2=0\\
\text{the Leibnitz rule~\eqref{eq:L}}\\
\text{the associativity}
\end{cases}.
\]
It is also important to note that $\p f=\left[\m_A,f\right]_G$, $d f=\left[\m^1_A,f\right]_G$ and $\delta f=\left[\m^2_A,f\right]_G$ for $f\in C^\bullet(A)$.
Let $\f_A$ be the cohomology class of $\m_A$ in $\T^2_{poly}(A)$, 
which is actually the same as the one of $\m^1_A$ in $\T^2_{poly}(A)$.
We will see later that, for a non-negatively graded connected smooth compact Calabi--Yau dg $k$-algebra $A$,
the element $\f_A$ is the dg analogue of a weighted homogeneous polynomial with an isolated singularity,
which is the initial data for the original Saito's theory of primitive forms.
Indeed, by a direct calculation, we obtain the following dg analogue of the ``Euler's identity" for $\f_A$.
\begin{prop}
For $X\in \T^\bullet_{poly}(A)$, we have
$d X=\left[\f_A,X\right]_G$.
In particular, we have 
\begin{equation}\label{eq:f and deg}
\f_A=\left[{\mathfrak deg}_A,\f_A\right]_G,
\end{equation}
where ${\mathfrak deg}_A$ is the cohomology class of the element $\deg_A=(\deg_A^p)_{p\ge 0}\in C^1(A)$ in $\T^1_{poly}(A)$ 
defined by $\deg_A^p:=0$ if $p\ne 1$ and $\deg_A^1(Ta_1):=\overline{a_1}\cdot a_1$.
\end{prop}
The ``Euler's identity"~\eqref{eq:f and deg} means that $\f_A$ is of degree one with respect to another grading 
structure on $C^\bullet(A)$ whose origin is the $\ZZ$-grading on $A$. 
This completely agrees with the fact that a weighted homogeneous polynomial 
in the original Saito's theory is considered as of degree two from the homological view point, 
e.g. when we discuss matrix factorizations, 
and of degree one for the compatibility with exponents.
\subsection{Hochishild homology}
For a dg $k$-algebra $A$, set
\begin{equation}
C_\bullet(A):=\coprod_{n\ge 0}A \otimes_k  (TA)^{\otimes n}.
\end{equation}
We shall denote an element of $A \otimes_k  (TA)^{\otimes n}$ by $a_0$ if $n=0$ and 
$a_0\otimes Ta_1\otimes \dots \otimes Ta_{n}$ if $n\ne 0$.
Define a $k$-linear map $d:C_\bullet(A)\longrightarrow C_{\bullet-1}(A)$ by $d(a_0):=d_A a_0$ and, for $n\ge 1$, 
\begin{multline*}
d(a_0\otimes Ta_1\otimes\dots \otimes Ta_n )
:=d_Aa_0\otimes Ta_1\otimes \dots \otimes Ta_{n}\\ 
+\sum_{i=1}^n(-1)^{s'_i}
a_0\otimes Ta_1\otimes \dots \otimes Ta_{i-1}\otimes T(d_Aa_i)\otimes 
Ta_{i+1}\otimes \dots \otimes Ta_n,
\end{multline*}
where $s'_i:=\sum_{m=0}^{i}(\overline{a_m}-1)$. 
Define another $k$-linear map $\delta:C_\bullet(A)\longrightarrow C_{\bullet-1}(A)$ by
$\delta(a_0):=0$ and, for $n\ge 1$, 
\begin{multline*}
\delta(a_0\otimes Ta_1\otimes \dots \otimes Ta_n )
:=(-1)^{\overline{a_0}}a_0a_1\otimes Ta_2\otimes \dots\otimes Ta_n\\
-\sum_{i=1}^{n-1}(-1)^{s'_i}
a_0\otimes Ta_1\otimes\dots \otimes Ta_{i-1}\otimes T(a_ia_{i+1})\otimes Ta_{i+2}\otimes 
\dots \otimes Ta_n\\
+(-1)^{\overline{a_n}+(\overline{a_n}-1)s'_{n-1}}
a_na_0\otimes Ta_1\otimes \dots \otimes Ta_{n-1}.
\end{multline*}
Define a $k$-linear map $B:C_\bullet(A)\longrightarrow C_{\bullet+1}(A)$ called the {\em Connes differential} by 
\begin{align*}
&B(a_0\otimes Ta_1\dots \otimes Ta_n )\\
:=&{\rm id}_A\otimes Ta_0\otimes \dots \otimes Ta_n
-(-1)^{\overline{a_0}-1}a_0\otimes T({\rm id}_A)\otimes Ta_1\otimes \dots \otimes Ta_n\\
+&\sum_{i=1}^{n}(-1)^{\overline{a_n}-1+s''_i}
{\rm id}_A\otimes Ta_i\otimes \dots \otimes Ta_n\otimes Ta_0\otimes \dots \otimes Ta_{i-1}\\
-&\sum_{i=1}^{n}(-1)^{\overline{a_i}-1+s''_i}
a_i\otimes T({\rm id}_A)\otimes Ta_{i+1}\dots \otimes Ta_n\otimes Ta_0\otimes \dots \otimes Ta_{i-1},
\end{align*}
where $s''_i:=\left(\sum_{m=0}^{i-1}(\overline{a_m }-1)\right)\left(\sum_{m=i}^{n-1}(\overline{a_m }-1)\right)$.
The following proposition is also well-known.
\begin{prop}
The $k$-linear maps
$d, \delta, B$ define differentials on $C_\bullet(A)$ and 
satisfy $d\delta+\delta d=0,\ dB+Bd=0,\ \delta B+B\delta=0$.
Moreover, the dg $k$-module $(C_\bullet(A),\p:=d+\delta)$ is isomorphic to 
the dg $k$-module $A\otimes^\LL_{A^e} A$ in the derived category of dg $k$-modules.
\end{prop}
\begin{defn}
The dg $k$-module $(C_\bullet(A),\partial)$ is called the Hochschild chain complex
of $A$, whose homology $H_{\bullet}(C_\bullet(A),\partial)$ is denoted by $H\!H_\bullet (A)$ and is 
called the Hochschild homology of $A$.
Denote by $\Omega_\bullet(A)$ the graded $k$-module $H_{\bullet}(C_\bullet(A),\delta)$.
\end{defn}
The pair $(\T^\bullet_{poly}(A),\Omega_\bullet(A))$ admits a structure of a {\em calculus algebra} 
(cf. Dolgushev--Tamarkin--Tsygan, Definition~3 in \cite{dtt:1}), namely, 
we have the algebraic structures on $(\T^\bullet_{poly}(A),\Omega_\bullet(A))$ in the next
proposition which directly follows, just by forgetting the differential $d$, from the result by Daletski--Gelfand--Tsygan \cite{dgt:1} (see also Dolgushev--Tamarkin--Tsygan, Section~3 in \cite{dtt:2}). 
To state the result, recall two $k$-linear maps $\iota_f, \L_f :C_\bullet(A)\longrightarrow C_\bullet(A)$ 
called the {\it contraction} and the {\it Lie derivative}.
For $f\in \G r_k  ((TA)^{\otimes p},A)$, they are defined in the following way$:$
\begin{align*}
&\iota_f(a_0\otimes Ta_1\dots \otimes Ta_n )\\
:=&\pm a_0f(Ta_{1}\otimes \dots \otimes Ta_p)\otimes Ta_{p+1}\dots \otimes Ta_{n},\\
&\L_f(a_0\otimes Ta_1\dots \otimes Ta_n )\\
:=
&f(Ta_0\otimes \dots \otimes Ta_{p-1} )\otimes \dots \otimes Ta_{n}\\
&+\sum_{i=1}^{n-p+1}
\pm a_0\otimes Ta_1\dots \otimes  Tf(Ta_{i}\otimes \dots \otimes Ta_{i+p-1})\otimes \dots \otimes Ta_{n}\\
&+\sum_{i=n-p+1}^{n-1}
\pm f(Ta_{i+1}\otimes \dots \otimes Ta_n\otimes Ta_0\otimes \cdots )\otimes \dots \otimes Ta_{i},
\end{align*}
where we omit signs since we do not need their explicit expressions later.
\begin{prop}\label{prop:iL}
The $k$-linear maps $\iota$ and $\L$ induce morphisms of graded $k$-modules
\[
i:\T^\bullet_{poly}(A)\longrightarrow \G r_k (\Omega_\bullet(A),\Omega_\bullet(A)),\quad X\mapsto i_X,
\]
\[
L:\T^{\bullet}_{poly}(A)\longrightarrow T^{-1}\G r_k (\Omega_\bullet(A),\Omega_\bullet(A)),\quad X\mapsto L_X,
\]
satisfying, for all $X,Y\in\T^\bullet_{poly}(A)$,
\[
i_X i_Y=i_{X\circ  Y},\  \left[L_X,L_Y\right]=L_{[X,Y]_G},
\]
\[
L_X i_Y+(-1)^{\overline{X}} i_Y L_X=L_{X\circ  Y},\ \left[ i_X, L_Y\right]=i_{[X,Y]_G},
\]
\[
\left[B, i_X\right]=-L_{X},\ \left[B,L_X\right]=0,\ L_{\f_A}=-d.
\]
\end{prop}
Note that $\G r_k (\Omega_\bullet(A),\Omega_\bullet(A))$ has both a structure of 
a dg $k$-algebra and a dg Lie algebra.
Since we have 
\[
[d, i_X]=i_{dX},\ [d, L_X]=L_{-dX},\quad X\in \T^\bullet_{poly}(A),
\]
the $k$-linear maps $i$ and $L$ define a morphism of dg $k$-algebras 
and a morphism of dg Lie algebras.
Since $A^e$-modules can be considered as dg endo-functors on the derived dg category of 
dg $A$-modules (see T\"oen, Corollary~7.6 in \cite{Toen} for the precise statement), 
there are ``horizontal" and ``vertical" associative product structures on $\RR{\rm Hom}_{A^e}(A,A)$ and 
$\RR{\rm Hom}_{A^e}(A,A)$-module structures on $\RR{\rm Hom}_{A^e}(A^!,A)$ 
(cf. Kashiwara--Schapira, Remark~1.3.4 in \cite{Kas-S}, for endo-functors on ordinary categories).
It turns out that they induce the same structures on $\T^\bullet_{poly}(A)$ and $\Omega_\bullet(A)$, 
which are the product $\circ$ and the map $i$ given above (cf. Kaledin, Lemma~8.1 and the equation~(8.2) in \cite{KaledinTokyo}). 
The equalities in Proposition~\ref{prop:iL}, the ``Cartan calculus",  play an essential role in our story,
which enable us to translate the original Saito's theory for isolated hypersurface singularities 
verbatim into the dg categorical Saito theory.
\section{Construction of a very good section}
\subsection{Formality for certain homotopy calculus algebras}
Let $A$ be a non-negatively graded smooth dg $k$-algebra.
On the algebraic structure on the pair $(C^\bullet(A),C_\bullet(A))$,
the following proposition directly follows from Corollary~1 in \cite{dtt:1}:
\begin{prop}
The pair $(C^\bullet(A),C_\bullet(A))$ has a structure of homotopy calculus algebra.
\end{prop}
We omit the definition of homotopy calculus algebras due to limitations of space since 
we shall not use it in later discussions. 
Based on Theorem~5 in \cite{dtt:1}, we expect the following conjecture$:$
\begin{conj}\label{conj:formality}
As homotopy calculus algebras, $(C^\bullet(A),C_\bullet(A))$ is quasi-isomorphic to 
the calculus algebra $(\T^\bullet_{poly}(A),\Omega_\bullet(A))$.
\end{conj}
Note in particular that 
Conjecture~\ref{conj:formality} implies that 
the dg Lie algebra $(C^{\bullet+1}(A),[\m_A,-]_G,[-,-]_G)$ is quasi-isomorphic to 
the dg Lie algebra $(\T^{\bullet+1}_{poly}(A),[\f_A,-]_G,[-,-]_G)$ as an $L_\infty$-algebra.
\begin{assum}\label{assum:formality}
Conjecture~\ref{conj:formality} holds for a non-negatively graded smooth dg $k$-algebra $A$.
\end{assum}
This assumption is important in order also to ensure the functoriality of our construction of primitive forms.
\subsection{Hochschild cohomology of Calabi--Yau dg algebras}
From now on, $A$ always denotes a non-negatively graded connected smooth compact Calabi--Yau dg $k$-algebra.
Under Assumption~\ref{assum:formality}, there are isomorphisms of graded $k$-modules
\[
H\!H^\bullet(A)\cong H^\bullet(\T^\bullet_{poly}(A),d),\quad 
H\!H_\bullet(A)\cong H_\bullet(\Omega_\bullet(A),d).
\]
In particular, an element of $H\!H_w(A)$ giving the isomorphism~\eqref{defn:CY} corresponds to
a non-zero element $v_1\in H_w(\Omega_\bullet(A),d)$. 
Since the dg $k$-algebra $A$ is non-negatively graded and connected, we have 
$H_w(\Omega_\bullet(A),d)\cong H\!H_w(A)\cong H\!H^0(A)= k\cdot [1_A]$, which implies
that $H_w(\Omega_\bullet(A),d)=k\cdot v_1$.
\begin{conj}\label{conj:isom}
Let $v_1$ be a non-zero element in $H_w(\Omega_\bullet(A),d)$.
Under Assumption~\ref{assum:formality}, the contraction map
\[
(C^\bullet(A),\p)\longrightarrow (C_\bullet(A),\p),\quad 
X\mapsto \iota_Xv_1,
\]
induces a morphism of dg $k$-modules 
\[
(\T^\bullet_{poly}(A),d)\longrightarrow (\Omega_{w-\bullet}(A),d),\quad X\mapsto i_Xv_1,
\]
which is an isomorphism.
\end{conj}
\begin{assum}\label{assum:isom}
Conjecture~\ref{conj:isom} holds for
a non-negatively graded connected smooth compact Calabi--Yau dg $k$-algebra $A$.
\end{assum}
In particular, under Assumption~\ref{assum:isom} we have the isomorphism 
\begin{equation}
H^p(\T^\bullet_{poly}(A),d)\cong H_{w-p}(\Omega_\bullet(A),d),\ p\in\ZZ, \quad X\mapsto i_X v_1,
\end{equation}
of graded $k$-modules.
\begin{defn}
Set 
\begin{equation}
Jac(\f_A):=H^\bullet(\T^\bullet_{poly}(A),d).
\end{equation}
We call the graded $k$-module $Jac(\f_A)$ the {\em Jacobian ring} of $A$.
\end{defn}
It is known that if $A$ is a smooth compact Calabi--Yau dg $k$-algebra of dimension $w$ then 
$A^e$ is a smooth compact Calabi--Yau dg $k$-algebra of dimension $2w$. 
Since $Jac(\f_A)\cong H\!H^\bullet(A)\cong H^{\bullet}(\RR{\rm Hom}_{A^e}(A,A))$, we have the following.
\begin{prop}
Fix a non-zero element $v_1\in H\!H_w(A)$.
Let $v_1^{\otimes 2}$ be the element in $H\!H_{2w}(A^e)$ 
corresponding to $v_1\otimes v_1$ under the K\"unneth formula $H\!H_{2w}(A^e)\cong H\!H_w(A^{op})\otimes_k H\!H_w(A)
=H\!H_w(A)\otimes_k H\!H_w(A)$. 
Then the induced isomorphism by $v_1^{\otimes 2}$
\[
T^{2w}A^e\longrightarrow (A^e)^*
\]
in the perfect derived category of the dg $k$-algebra $(A^e)^e$ equips $Jac(\f_A)$ with 
a structure of a finite dimensional graded commutative Frobenius $k$-algebra.
Namely, we have a non-degenerate graded symmetric bilinear form
$\eta_{\f_A,v_1^{\otimes 2}}:Jac(\f_A)\otimes_k  Jac(\f_A)\to T^{-2w}k$
such that  
\begin{equation}
\eta_{\f_A,v_1^{\otimes 2}}(X\circ Y,Z)=\eta_{\f_A,v_1^{\otimes 2}}(X,Y\circ Z),\quad X,Y,Z\in Jac(\f_A).
\end{equation}
\end{prop}
\subsection{Filtered de Rham cohomology and the degeneration of Hodge to de Rham}
\begin{defn}
Let $u$ be a formal variable of degree two.
Define a graded $k((u))$-module $\H_{\f_A}$, called the filtered de Rham cohomology, by
\begin{equation}
\H_{\f_A}:=H_\bullet(T^{-w}\Omega_\bullet (A)((u)),d +uB)
\end{equation}
and for any integer $p\in\ZZ$ the graded $k[[u]]$-submodules $\H_{\f_A}^{(-p)}$of $\H_{\f_A}$
\begin{equation}
\H_{\f_A}^{(-p)}:=H_\bullet(T^{-w}\Omega_\bullet (A)[[u]]u^{p},d +uB).
\end{equation}
Define a graded $k$-module $\Omega_{\f_A}$ by
\begin{equation}
\Omega_{\f_A}:=H_\bullet(T^{-w}\Omega_\bullet(A),d).
\end{equation}
\end{defn}
\begin{prop}\label{prop:Hodge to de Rham}
For all $p\in\ZZ$, there exists an exact sequence of graded $k$-modules
\begin{equation}
0\longrightarrow \H_{\f_A}^{(-p-1)}\longrightarrow 
\H_{\f_A}^{(-p)} \stackrel{r^{(-p)}}{\longrightarrow }
\Omega_{\f_A} \longrightarrow 0.
\end{equation}
\end{prop}
\begin{proof}
For non-negatively graded smooth compact dg $k$-algebras, 
it follows that ${\rm Im}(B)\cap H\!H_\bullet(A)=0$ from the degeneration of Hodge to de Rham conjecture proven by 
Kaledin, Theorem~5.5 in \cite{kaledin:1}.
Under Assumption~\ref{assum:formality} we have $H\!H_\bullet(A)\cong T^{-w}\Omega_{\f_A}$, which implies 
${\rm Im}(B)\cap \Omega_{\f_A}=0$ and hence $r^{(-p)}$ is surjective. 
\end{proof}
The $k[[u]]$-submodules $\left\{\H_{\f_A}^{(-p)}\right\}_{p\in\ZZ}$ of $\H_{\f_A}$ define an increasing filtration 
\[
\dots\subset \H_{\f_A}^{(-p-1)}\subset \H_{\f_A}^{(-p)}\subset\dots \subset 
\H_{\f_A},   
\]
such that the multiplication of $u$ induces an isomorphism of $k$-modules
\[
u:\H_{\f_A}^{(-p)}\cong \H_{\f_A}^{(-p-1)}.
\]
\subsection{Gau\ss--Manin connection on $\H_{\f_A}$}
Set $\T_{\widehat{\AA}^1_{u}}:=k[[u]]\frac{d}{du}$, where $\frac{d}{du}$ is the derivation on $k[[u]]$ satisfying $\frac{d}{du}(u)=1$.
Define a morphism of graded $k$-modules
$\nabla:\T_{\widehat{\AA}^1_{u}}\otimes_k  
\Omega_\bullet (A)((u))\to \Omega_\bullet (A)((u))$ by 
\begin{equation}
\nabla_{\frac{d }{du}}:=\frac{d }{du}-\frac{1}{u^2} i_{\f_A}.
\end{equation}
The following proposition is very simple, however, it will be the key in our construction of a primitive form.
It is the dg analogue of 
the corresponding fact used in the original Saito's theory for weighted homogeneous polynomials 
when we lift a homogeneous basis of the Jacobian ring multiplied with the standard holomorphic volume form 
to a very good section. 
\begin{prop}
The morphism of graded $k$-modules $\nabla$ is a connection which satisfies
\[
\left[\nabla_{u\frac{d}{du}},d+uB\right]=d+uB.
\]
Therefore, $\nabla$ induces a connection on $\H_{\f_A}$.
Moreover, we have 
\[
\nabla_{u\frac{d}{du}}\left(\H_{\f_A}^{(0)} \right)\subset \H_{\f_A}^{(0)}.
\]
\end{prop}
\begin{proof}
Use equalities in Proposition~\ref{prop:iL}. We have 
\[
\left[\nabla_{u\frac{d}{du}},d+uB\right]=uB-\frac{1}{u}\left[i_{\f_A},d+uB \right]
=uB-L_{\f_A}=d+uB
\]
Together with the equation~\eqref{eq:f and deg}, we also have 
\[
u\frac{d}{du}-\frac{1}{u}i_{\f_A}=u\frac{d}{du}-\frac{1}{u}\left[i_{{\mathfrak deg}_A},L_{\f_A}\right]
=u\frac{d}{du}+L_{{\mathfrak deg}_A}+\frac{1}{u}\left[d+uB, i_{{\mathfrak deg}_A}\right]
\]
\end{proof}
\begin{defn} 
The connection $\nabla$ on $\H_{\f_A}$ is called the Gau\ss--Manin connection for $A$.
\end{defn}
\subsection{Exponents}
For $a_0\otimes Ta_1\otimes\dots \otimes Ta_n\in C_\bullet(A)$ such that $a_i\in A^{\overline{a_i}}$,  
the endomorphism $\L_{\deg_A}$ on $C_\bullet(A)$ can be calculated as 
\[
\L_{\deg_A}\left(a_0\otimes Ta_1\otimes\dots \otimes Ta_n \right):=
\left(\sum_{i=0}^n\overline{a_i}\right)\cdot \left(a_0\otimes Ta_1\otimes\dots \otimes Ta_n \right).
\]
This obviously commutes with the operator $\delta$ on $C_\bullet(A)$ and hence defines 
an endomorphism on $\Omega_\bullet(A)$,  which is exactly $L_{{\mathfrak deg}_A}$.
\begin{prop}\label{prop:degree endo}
The endomorphism of graded $k$-modules $L_{{\mathfrak deg}_A}$ on $\Omega_\bullet(A)$ induces a $k$-linear endomorphism
on $\Omega_{\f_A}$. 
\end{prop}
\begin{proof}
The equalities in Proposition~\ref{prop:iL} and the equation~\eqref{eq:f and deg} yield
\[
\left[L_{{\mathfrak deg}_A},d\right]=-\left[L_{{\mathfrak deg}_A},L_{\f_A}\right]=-L_{[{\mathfrak deg}_A,\f_A]_G}=-L_{\f_A}=d,
\]
which implies that $L_{{\mathfrak deg}_A}(\Omega_{\f_A})\subset \Omega_{\f_A}$.
\end{proof}

Denote by $N_A$ the endomorphism of graded $k$-modules on $\Omega_{\f_A}$ induced by $L_{{\mathfrak deg}_A}$.
Define graded $k$-submodules $\Omega_{\f_A}^{p,q}$ of $\Omega_{\f_A}$ by 
\[
\Omega_{\f_A}^{p,q}:=\{\omega\in \Omega_{\f_A}~|~\overline{\omega}=w-p+q,\ N_A\omega=q\omega\},\quad p,q\in\ZZ.
\]
and set
\[
F_q:=\bigoplus_{p\in\ZZ, r\le q}\Omega_{\f_A}^{p,r},\quad q\in \ZZ.
\]
\begin{defn}
The graded $k$-submodules $\{F_q\}_{q\in\ZZ}$ form an increasing filtration of $\Omega_{\f_A}$
\begin{equation}
0\subset \dots \subset F_q\subset F_{q+1}\subset \dots\subset \Omega_{\f_A},
\end{equation}
which is called the Hodge filtration of $\Omega_{\f_A}$.
\end{defn}
Since $A$ is non-negatively graded, we have $F_{q}=0$ for $q<0$. 
\begin{defn}
The Hodge numbers are  
\[
h^{p,q}(A):= \dim_k   \Omega_{\f_A}^{p,q},\quad p,q\in\ZZ.
\]
The integer $q$ with $h^{p,q}(A)\ne 0$ is called an exponent.
The set of exponents is the multi-set 
\[
\left\{q*h^{p,q}(A)~|~p,q\in\ZZ,\ \Omega_{\f_A}^{p,q}\ne 0 \right\},
\]
where by $u*v$ we denote $v$ copies of the integer $u$.  
\end{defn}
\subsection{Pairing on the Hochschild homology}
Since $A$ is Calabi--Yau of dimension $w$, an element $v_1\in H\!H_w(A)$ yields an isomorphism of graded $k$-modules 
\begin{equation}
Jac(\f_A)\cong \Omega_{\f_A},\quad X\mapsto i_X v_1.
\end{equation}
Move the $k$-bilinear form on $Jac(\f_A)$ to $\Omega_{\f_A}$ by this isomorphism.
\begin{defn}
Define a $k$-bilinear form $J_{\f_A}:\Omega_{\f_A}\otimes_k  \Omega_{\f_A}\to T^{-2w}k$ by 
\begin{equation}
J_{\f_A}(i_Xv_1,i_Y v_1):=(-1)^{w\cdot \overline{Y}}\eta_{\f_A,v_1^{\otimes 2}}(X,Y),\quad X,Y\in Jac(\f_A).
\end{equation}
\end{defn}
Note that $J_{\f_A}$ does not depend on the choice of $v_1$. 
Moreover, it induces a perfect pairing 
\[
J_{\f_A}: \Omega_{\f_A}^{p,q}\otimes_k  \Omega_{\f_A}^{w-p,w-q}\longrightarrow k.
\]
Some elementary homological algebras for
$Jac(\f_A)$ and $\Omega_{\f_A}$ yields the following.
\begin{prop}\label{prop:Hodge numbers}
The Hodge numbers satisfy
\begin{enumerate}
\item
$h^{p,q}(A)=0$ if $p < 0$ or $q<0$.
\item
$h^{w,0}(A)=1$.
\item
$h^{w-p,q}(A)=h^{p,w-q}(A)$.
\end{enumerate}
\end{prop}
\begin{proof}
The property {\rm (1)} holds since $A$ is non-negatively graded.
The property {\rm (2)} follows from the fact that $[1_A]\in H\!H^0(A)$ is mapped to $v_1$ 
under the isomorphism $H\!H^0(A)\cong H\!H_w(A)$ and that $v_1\in \Omega_{\f_A}^{w,0}$.
The perfectness of $J_{\f_A}$ implies the property {\rm (3)}.
\end{proof}
This proposition means that $0$ is the minimal exponent and 
the duality of exponents holds, namely, if $q$ is an exponent then $w-q$ is also an exponent.
\subsection{Existence of a very good section}
\begin{prop}\label{prop:very good section}
Let $r^{(0)}:\H_{\f_A}^{(0)} \longrightarrow \Omega_{\f_A}$ be the $k$-linear map of degree zero in the exact sequence 
of graded $k$-modules in Proposition~\ref{prop:Hodge to de Rham}$:$
\[
0\longrightarrow \H_{\f_A}^{(-1)}\longrightarrow 
\H_{\f_A}^{(0)} \stackrel{r^{(0)}}{\longrightarrow }
\Omega_{\f_A} \longrightarrow 0.
\]
Then there exists a section $s^{(0)}:\Omega_{\f_A}\longrightarrow \H_{\f_A}^{(0)}$ of $r^{(0)}$ such that  
\[
\nabla_{u\frac{d}{du}}\left(s^{(0)}\left(\Omega_{\f_A}\right)\right)\subset s^{(0)}\left(\Omega_{\f_A}\right).
\]
\end{prop}
\begin{proof}
Set $l_A:=\dim_k  Jac(\f_A)$. Then 
Proposition~\ref{prop:degree endo} implies the existence of a $k$-basis $\{v_1,\dots, v_{l_A}\}$ of $\Omega_{\f_A}$
such that $N_Av_i=q_i\cdot v_i$ for some $q_i\in\ZZ$ for $i=1,\dots ,l_A$, where 
$v_1$ is the element in $\Omega_{\f_A}^{w,0}$ as in the previous section.
There exists a section $s^{(0)}:\Omega_{\f_A}\longrightarrow \H_{\f_A}^{(0)}$ so that $s^{(0)}\left(\Omega_{\f_A}\right)$ is inside of 
$H_\bullet(T^{-w}\Omega_\bullet (A)\otimes_k  k[u],d +uB)$ since $k$ is a field and the dg $k$-algebra $A$ is non-negatively graded. 
Moreover, the equalities $\left[\nabla_{u\frac{d}{du}},d+uB\right]=d+uB$, 
$\left[L_{{\mathfrak deg}_A},d\right]=d$ and $\left[L_{{\mathfrak deg}_A},B\right]=0$, which in particular 
means that $d$ (resp. $B, u$) is of degree $1$ (resp. $0,1$) with respect to the grading given by $L_{{\mathfrak deg}_A}$, enable us to choose $\omega_{il}\in\Omega_\bullet(A)$ 
satisfying  $L_{{\mathfrak deg}_A}\omega_{i,l}=(q_i-l)\cdot \omega_{i,l}$ 
so that $s^{(0)}(v_i)=\displaystyle\sum_{l=0}^m [\omega_{i,l}] u^l$ for some $m\in\ZZ_{\ge 0}$.
Therefore, $\nabla_{u\frac{d}{du}}s^{(0)}(v_i)=q_i\cdot s^{(0)}(v_i)$.
\end{proof}
A section $s^{(0)}$ of $r^{(0)}$ in Proposition~\ref{prop:very good section} is called a {\it very good section} after 
M.~Saito (see Introduction and Proposition 3.1~ in \cite{MSaito2}).
\subsection{Higher residue pairings}
Once a very good section $s^{(0)}$ for $r^{(0)}$ is given, we have an isomorphism 
\[
\Omega_{\f_A}((u))\cong \H_{\f_A},\quad 
 \sum_{p=-\infty}^\infty v_p u^{p}\mapsto 
\sum_{p=-\infty}^\infty s^{(0)}(v_p)u^{p},
\]
of $k((u))$-modules.
\begin{defn}
Fix a very good section $s^{(0)}$ for $r^{(0)}$ and define a $k$-bilinear form
$K_{\f_A}:\H_{\f_A}\otimes_k \H_{\f_A}\to k((u))$ by 
\begin{equation}
K_{\f_A}\left(\omega(u),\omega'(u)\right):=J_{\f_A}\left(\omega(u),\omega'(-u)\right)u^w,
\end{equation}
where $J_{\f_A}$ on the right hand sided denotes the pairing on 
$\Omega_{\f_A}((u))$ which is the $k((u))$-linear
extension of $J_{\f_A}$ on $\Omega_{\f_A}$.
The pairing $K_{\f_A}$ is called the higher residue pairings after K.~Saito, Section~4 in \cite{s:2}.
\end{defn}
It follows from some elementary calculation that for all $h(u)\in k((u))$
\[
h(u)K_{\f_A}(\omega_1,\omega_2)=K_{\f_A}(h(u)\omega_1,\omega_2)=K_{\f_A}(\omega_1,h(-u)\omega_2),
\]
and for all $\omega_1, \omega_2\in {\H}_{\f_A}$  
\[
u\frac{d}{du}K_{\f_A}(\omega_1,\omega_2)=K_{\f_A}(\nabla_{u\frac{d}{du}}\omega_1,\omega_2)
+K_{\f_A}(\omega_1,\nabla_{u\frac{d}{du}}\omega_2).
\]
Since these properties characterize uniquely the higher residue pairings for isolated hypersurface singularities 
(see p.45 -- p.46 in \cite{MSaito}), we just denote it by $K_{\f_A}$ without mentioning the dependence on 
the chosen very good section $s^{(0)}$.
Proposition~\ref{prop:very good section} imply the following.
\begin{prop}\label{prop:marking}
Let $s^{(0)}$ be a very good section for $r^{(0)}$.
The graded $k$-submodule $S:=s^{(0)}(\Omega_{\f_A})\otimes_kk[u^{-1}] u^{-1}$ of $\H_{\f_A}$ 
satisfies 
\[
\H_{\f_A}=\H_{\f_A}^{(0)}\oplus S,
\ u^{-1}S\subset S,
\ \nabla_{u\frac{d}{du}}S\subset S,
\ K_{\f_A}(S,S)\subset k[u^{-1}]u^{w-2}.
\]
\end{prop}
\section{Deformation}
In this section, $A$ always denotes a non-negatively graded smooth compact Calabi--Yau dg $k$-algebra of dimension $w$ satisfying Assumption~\ref{assum:formality} and Assumption~\ref{assum:isom}.
\subsection{Versal deformation}
Recall that $\Omega_{\f_A}^{w,0}=\{v\in \Omega_w(A)~|~dv=0, N_Av=0\}=k\cdot v_1$.
Under Assumption~\ref{assum:isom}, 
we can define a morphism $\Delta:\T^\bullet_{poly}(A)\longrightarrow \T^{\bullet-1}_{poly}(A)$ of graded $k$-modules by 
$i_{\Delta(X)} v_1:=B i_X v_1$, $X\in \T^\bullet_{poly}(A)$. 
Note that $\Delta$ does not depend on the particular choice of $v_1$.
Obviously, it satisfies $\Delta^2=0$ and hence it defines a differential on $\T^\bullet_{poly}(A)$.
The differential $\Delta$ does not satisfy the Leibniz rule with respect to the product $\circ$,
however, it is a part of a rich structure on $\T^\bullet_{poly}(A)$ as given below.
\begin{prop}
The tuple $(\T^\bullet_{poly}(A),d,\circ, [-,-]_G,\Delta)$ is a dGBV algebra.
Namely, we have
\[
\left[X,Y\right]_G=(-1)^{\overline{X}}\Delta(X\circ Y)-(-1)^{\overline{X}}\Delta(X)\circ Y-X\circ \Delta(Y),\ 
X,Y\in \T^\bullet_{poly}(A).
\]
\end{prop}
\begin{proof}
The equality
\begin{align*}
&i_{\left[X,Y\right]_G}v_1
=\ \left[i_X,L_Y\right]v_1\\
=&\ -i_X\left[B,i_Y\right]v_1+(-1)^{\overline{X}\cdot (\overline{Y}+1)}\left[B,i_Y\right]i_Xv_1\\
=&\ -i_X B i_Y v_1+(-1)^{\overline{X}\cdot (\overline{Y}+1)}B i_Y i_Xv_1
-(-1)^{\overline{Y}}\cdot (-1)^{\overline{X}\cdot (\overline{Y}+1)}i_YB i_Xv_1\\
=&\ -i_{X\circ \Delta(Y)}v_1+(-1)^{\overline{X}\cdot (\overline{Y}+1)}i_{\Delta(Y\circ X)}v_1
+(-1)^{(\overline{X}+1)\cdot (\overline{Y}+1)}i_{Y\circ\Delta(X)}v_1\\
=&\ -i_{X\circ \Delta(Y)}v_1+(-1)^{\overline{X}}i_{\Delta(X\circ Y)}v_1
-(-1)^{\overline{X}}i_{\Delta(X)\circ Y}v_1
\end{align*}
yields the statement.
\end{proof}
Denote by $\O_M$ the completed symmetric algebra $k[[T^2Jac(\f_A)]]$ of $T^2Jac(\f_A)$ 
and by $\m$ the maximal ideal in $k[[T^2Jac(\f_A)]]$.
Note that $\O_M$ is isomorphic to the completed symmetric algebra $k[[HH^{\bullet+2}(A)]]$ of $HH^{\bullet+2}(A)$ 
under Assumption~\ref{assum:formality}.
Let $t_1,\dots, t_{l_A}$ be the dual coordinates for the basis $\{v_1,\dots, v_{l_A}\}$ as in the proof of 
Proposition~\ref{prop:very good section}. Denote by $\T_M$ a graded $\O_M$-free module of derivations on $\O_M$, 
which satisfies 
$\T_M\cong \bigoplus_{i=1}^{l_A}\O_M\frac{\p}{\p t_i}$.
\begin{prop}\label{prop:versal deformation}
The dGBV algebra $(\T^\bullet_{poly}(A),d,\circ, [-,-]_G,\Delta)$ is smooth formal.
Namely, there exists a solution $\gamma(t)$ to the Maurer--Cartan equation in formal power series with values in $\T^\bullet_{poly}(A)$,
\begin{equation}\label{eq:Maurer--Cartan}
d\gamma(t)+\frac{1}{2}\left[\gamma(t),\gamma(t)\right]_G=0,\quad 
\gamma(t)\in \T^\bullet_{poly}(A)\widehat{\otimes}_k  \m,
\end{equation}
satisfying the following properties$:$ 
\begin{enumerate}
\item 
For $i=1,\dots, l_A$, the element
\[
\left[\frac{\p \gamma(t)}{\p t_i}\right]\in 
\T^\bullet_{poly}(A)\widehat{\otimes}_k  \m\left/\T^\bullet_{poly}(A)\widehat{\otimes}_k  \m^2\right.
\]
considered as an element in $\T^\bullet_{poly}(A)$ form a $k$-basis of $Jac(\f_A)$.
\item 
The solution $\gamma(t)$ is homogeneous in the sense that
\begin{equation}\label{eq:gamma-hom}
\gamma(t)=\sum_{i=1}^{l_A}(1-q_i)t_i\frac{\p \gamma(t)}{\p t_i}+\left[{\mathfrak deg}_A,\gamma(t)\right]_G.
\end{equation} 
\end{enumerate}
\end{prop}
\begin{proof}
We can apply the Terilla's result, Theorem~2 in \cite{ter:1}, since our dGBV algebra 
$(\T^\bullet_{poly}(A),d,\circ, [-,-]_G,\Delta)$ satisfies his ``degeneration of the spectral sequence" condition
due to the Hodge to de Rham degeneration of our filtered de Rham cohomology $\H_{\f_A}$ 
(Proposition~\ref{prop:Hodge to de Rham}).
\end{proof}
Let $\gamma(t)$ be as in Proposition~\ref{prop:versal deformation}.
Define a formal power series $\ff_A$ with values in $\T^\bullet_{poly}(A)$ as
\begin{equation}
\ff_A:=\f_A\otimes 1+\gamma(t).
\end{equation}
Then it follows that $\left[\ff_A,\ff_A\right]_G=0$ from the Maurer--Cartan equation~\eqref{eq:Maurer--Cartan}.
Define an  $\O_M$-endomorphism $d_\gamma$ on $\T^\bullet_{poly}(A)\widehat{\otimes}_k \O_M$ by 
\begin{equation}
d_\gamma X:=\left[\ff_A,X\right]_G,\quad X\in \T^\bullet_{poly}(A)\widehat{\otimes}_k  \O_M. 
\end{equation}
It follows that $d_\gamma^2=0$ since $\left[\ff_A,\ff_A\right]_G=0$.
\begin{prop}
For all $X,Y\in \T^\bullet_{poly}(A)\widehat{\otimes}_k \O_M$, we have 
\[
d_\gamma\left[X, Y\right]_G=\left[d_\gamma(X),Y\right]_G+(-1)^{\overline{X}}\left[X,d_\gamma (Y)\right]_G, 
\]
\[
d_\gamma(X\circ Y)=d_\gamma (X)\circ Y+(-1)^{\overline{X}}X\circ d_\gamma (Y). 
\]
Namely, the triple $(d_\gamma, \circ, [-,-]_G)$ equip $\T^\bullet_{poly}(A)\widehat{\otimes}_k  \O_M$ with 
a structure of differential Gerstenharber algebra.
\end{prop}
\begin{proof}
The statement follows since the bracket $[-,-]_G$ satisfies the Jacobi and the Poisson identities. 
\end{proof}
\begin{defn}
The graded $\O_M$-module
\begin{equation}
Jac(\ff_A):=H^\bullet(\T^\bullet_{poly}(A)\widehat{\otimes}_k  \O_M,d_\gamma),
\end{equation}
is called the Jacobian ring of $\ff_A$.
\end{defn}
Note that the property $\left[\ff_A,\ff_A\right]_G=0$ implies 
\[
d_\gamma\left(\ff_A\right)=0,
\quad d_\gamma\left(\frac{\p \ff_A}{\p t_i}\right)=\left[\ff_A,\frac{\p \ff_A}{\p t_i}\right]_G=0,\  i=1,\dots, l_A.
\]
\begin{prop}
The morphism $\rho:\T_M\longrightarrow Jac(\ff_A)$ of graded $\O_M$-modules defined by
\begin{equation}\label{eq:KS}
\rho:\T_M\longrightarrow Jac(\ff_A),\quad \frac{\p}{\p t_i}\mapsto \left[\frac{\p \ff_A}{\p t_i}\right],\ i=1,\dots, l_A,
\end{equation} 
is an isomorphism.
\end{prop}
The isomorphism~\eqref{eq:KS} enable us to introduce two particular elements of $\T_M$ which play important roles later.
\begin{defn}
The element $e\in\T_M$ such that $\rho(e)=[1_A]$ is called the primitive vector field.
The element $E\in\T_M$ such that $\rho(E)=[\ff_A]$ is called the Euler vector field.
\end{defn}
The equation~\eqref{eq:gamma-hom} in Proposition~\ref{prop:versal deformation} implies the following.
\begin{prop}\label{prop:Euler}
We have 
\begin{equation}
E=\sum_{i=1}^{l_A}(1-q_i)t_i\frac{\p }{\p t_i}
\end{equation}
and the ``Euler's identity":
\begin{equation} 
\ff_A=E\ff_A+\left[{\mathfrak deg}_A,\ff_A\right]_G.
\end{equation} 
\end{prop}
Consider the $\O_M$-linear extensions of the morphisms $i$ and $L$ of 
graded $k$-modules defined in Proposition~\ref{prop:iL} and 
define a morphism $d_\gamma$ of graded $\O_M$-modules on $\Omega_\bullet(A)\widehat{\otimes}_k \O_M$ as 
\begin{equation} 
d_\gamma:=-L_{\ff_A},
\end{equation} 
which is a deformation by $\gamma$ of the boundary operator $d$ on $\Omega_\bullet(A)$. 
The equalities in Proposition~\ref{prop:iL} lead the following.
\begin{prop}\label{prop:46}
We have 
\[
d_\gamma^2=0,\quad\left[B,i_{\ff_A}\right]=-L_{\ff_A}=d_{\gamma},\quad \left[B,d_\gamma\right]=0,
\]
\[
\left[d_\gamma,i_X\right]=i_{d_\gamma X},\ X\in \T^{\bullet}_{poly}(A)\widehat{\otimes}_k \O_M.
\]
In particular, $d_\gamma$ defines a boundary operator on $\Omega_\bullet(A)\widehat{\otimes}_k \O_M$. 
\end{prop}
\subsection{Deformed filtered de Rham cohomology $\H_{\ff_A}$}
\begin{defn}
Let $u$ be a formal variable of degree two.
Define a graded $k((u))\widehat{\otimes}\O_M$-module, called the {\it deformed filtered de Rham cohomology}, by
\begin{equation}
\H_{\ff_A}:=H_\bullet(T^{-w}\Omega_\bullet (A)((u))\widehat{\otimes}_k \O_M,d_\gamma+uB)
\end{equation}
and for any integer $p\in\ZZ$ the graded $\O_M[[u]]$-submodules of $\H_{\ff_A}$
\begin{equation}
\H_{\ff_A}^{(-p)}:=H_\bullet(T^{-w}\Omega_\bullet (A)[[u]]u^{p}\widehat{\otimes}_k  \O_M,d_\gamma +uB).
\end{equation}
Define a graded $\O_M$-module $\Omega_{\ff_A}$ by
\begin{equation}
\Omega_{\ff_A}:=H_\bullet(T^{-w}\Omega_\bullet(A)\widehat{\otimes}_k  \O_M,d_\gamma).
\end{equation}
\end{defn}
Proposition~\ref{prop:Hodge to de Rham} and Nakayama's Lemma imply the following. 
\begin{prop}
For all $p\in\ZZ$, there exists an exact sequence of graded $\O_M$-modules
\[
0\longrightarrow \H_{\ff_A}^{(-p-1)}\longrightarrow 
\H_{\ff_A}^{(-p)} \stackrel{}{\longrightarrow }
\Omega_{\ff_A} \longrightarrow 0.
\]
\end{prop}
\subsection{Gau\ss--Manin connection on $\H_{\ff_A}$}
Set
$\T_{\widehat{\AA}^1_{u}\times M}:=\O_M[[u]]\frac{d}{du}\oplus \O_M[[u]]\otimes_{\O_M}\T_M$ and
define a morphism of graded $k$-modules
$\nabla^\gamma:\T_{\widehat{\AA}^1_{u}\times M}\otimes_k  \left(
\Omega_\bullet (A)\widehat{\otimes}_k \O_M ((u))\right)\to \Omega_\bullet (A)\widehat{\otimes}_k \O_M((u))$ by 
\[
\nabla^\gamma_{\frac{d }{du}}:=\frac{d }{du}-\frac{1}{u^2} i_{\ff_A},\quad
\nabla^\gamma_{\frac{\p}{\p t_i}}:=\frac{\p}{\p t_i}+\frac{1}{u} i_{\frac{\p \ff_A}{\p t_i}},\ i=1,\dots, l_A.
\]
\begin{prop}
The morphism of graded $k$-modules $\nabla^\gamma$ is a flat connection which satisfies
\[
\left[\nabla^\gamma_{u\frac{d}{du}},d_\gamma+uB\right]=d_\gamma+uB,\quad
\left[\nabla^\gamma_{\frac{\p}{\p t_i}},d_\gamma+uB\right]=0,\ i=1,\dots, l_A.
\]
Therefore, $\nabla^\gamma$ induces a connection on $\H_{\ff_A}$.
Moreover, we have 
\[
\nabla^\gamma_{\frac{\p}{\p t_i}}\left(\H_{\ff_A}^{(-1)} \right)\subset \H_{\ff_A}^{(0)},\ i=1,\dots, l_A,\quad
\nabla^\gamma_{u\frac{d}{du}+E}\left(\H_{\ff_A}^{(0)} \right)\subset \H_{\ff_A}^{(0)}.
\]
\end{prop}
\begin{proof}
It is clear that $\nabla^\gamma$ is a connection, which is flat since 
{\small
\[
\left[\nabla^\gamma_{\frac{d}{du}},\nabla^\gamma_{\frac{\p}{\p t_i}}\right]=\frac{1}{u^2}i_{\frac{\p \ff_A}{\p t_i}}
-\frac{1}{u^2}i_{\frac{\p \ff_A}{\p t_i}}-\frac{1}{u^3}\left[i_{\ff_A},i_{\frac{\p \ff_A}{\p t_i}}\right]=0,
\]
\[
\left[\nabla^\gamma_{\frac{\p}{\p t_i}},\nabla^\gamma_{\frac{\p}{\p t_j}}\right]=
\frac{1}{u}i_{\frac{\p^2 \ff_A}{\p t_i\p t_j}}-(-1)^{(-\overline{t_i})\cdot (-\overline{t_j})}\frac{1}{u}
i_{\frac{\p^2 \ff_A}{\p t_j\p t_i}}+\frac{1}{u^2}\left[i_{\frac{\p \ff_A}{\p t_i}},i_{\frac{\p \ff_A}{\p t_j}}\right]=0.
\]
}
The equalities in Proposition~\ref{prop:46} yield 
{\small 
\begin{align*}
&\left[\nabla^\gamma_{u\frac{d}{du}},d_\gamma+uB\right]=uB-\frac{1}{u}\left[i_{\ff_A},d_\gamma+uB \right]
=uB-L_{\ff_A}=d_\gamma+uB,\\
&\left[\nabla^\gamma_{\frac{\p}{\p t_i}},d_\gamma+uB\right]=-\frac{\p }{\p t_i}L_{\ff_A}
-\frac{1}{u}\left[i_{\frac{\p \ff_A}{\p t_i}},L_{\ff_A} \right]+\left[i_{\frac{\p \ff_A}{\p t_i}},B \right]\\
=&\frac{\p }{\p t_i}\left[B,i_{\ff_A}\right]+\left[i_{\frac{\p \ff_A}{\p t_i}},B \right]
= (-1)^{-\overline{t_i}}\left[B,i_{\frac{\p \ff_A}{\p t_i}}\right]+(-1)^{1-\overline{t_i}}\left[B,i_{\frac{\p \ff_A}{\p t_i}}\right]=0.
\end{align*}
}
It is obvious from the definition that $\nabla^\gamma_{\frac{\p}{\p t_i}}\left(\H_{\ff_A}^{(-1)} \right)\subset \H_{\ff_A}^{(0)}$, $i=1,\dots, l_A$.
By Proposition~\ref{prop:Euler} and the equalities in Proposition~\ref{prop:46}, 
\begin{align*}
&\ u\frac{d}{du}-\frac{1}{u}i_{\ff_A}+E+\frac{1}{u}i_{E\ff_A}\\
=&\ u\frac{d}{du}+E-\frac{1}{u}\left[i_{{\mathfrak deg}_A},L_{\ff_A}\right]\\
=&\ u\frac{d}{du}+E+L_{{\mathfrak deg}_A}+\frac{1}{u}\left[d_\gamma+uB, i_{{\mathfrak deg}_A}\right],
\end{align*}
which gives the last statement.
\end{proof}
\begin{defn} 
The connection $\nabla^\gamma$ on $\H_{\ff_A}$ is called the Gau\ss--Manin connection.
\end{defn}
\section{Primitive forms and Frobenius structures}
Let $s^{(0)}:\Omega_{\f_A}\longrightarrow \H_{\f_A}^{(0)}$ be a very good section and 
let $\{v_1,\dots, v_{l_A}\}$ be elements of $\Omega_{\f_A}$ as in the proof of Proposition~\ref{prop:very good section}. 
Set $\zeta_i:=s^{(0)}(v_i)$ for $i=1,\dots, l_A$. 
\subsection{Fundamental solution to the Gau\ss--Manin connection}
By the equalities in Proposition~\ref{prop:iL} and Proposition~\ref{prop:46}
we obtain the following.
\begin{prop}
For all $\omega\in T^{-w}\Omega_\bullet(A)((u))\widehat{\otimes}_k \O_M$, we have
\[
(d_\gamma+uB)\left(e^{-\frac{i_{\gamma(t)}}{u}}\omega\right)=e^{-\frac{i_{\gamma(t)}}{u}}\left(d+uB\right)\omega
\]
Moreover, we have
\[
\nabla^\gamma_{\frac{\p}{\p t_j}}\left[e^{-\frac{i_{\gamma(t)}}{u}}\zeta_i\right]=0,\quad i,j=1,\dots, l_A,
\]
\[
\nabla^\gamma_{u\frac{d}{du}}\left[e^{-\frac{i_{\gamma(t)}}{u}}\zeta_i\right]
=q_i\cdot \left[e^{-\frac{i_{\gamma(t)}}{u}}\zeta_i\right], \quad i=1,\dots, l_A,
\]
where $\left[e^{-\frac{i_{\gamma(t)}}{u}}\zeta_i\right]$ denotes the equivalence class of 
$e^{-\frac{i_{\gamma(t)}}{u}}\zeta_i$ in $\H_{\ff_A}$.
\end{prop}
Therefore, we may identify $\H_{\f_A}\widehat{\otimes}_k \O_M$ with $\H_{\ff_A}$ as follows 
(cf. Proposition 1.3~in \cite{MSaito2}).
\begin{prop}
There is an isomorphism 
\begin{equation}\label{eq:flat sections}
\J_{\ff_A}:{\H}_{\ff_A}\stackrel{\cong}{\longrightarrow} {\H}_{\f_A}\widehat{\otimes}_k \O_M,
\quad \omega\mapsto e^{\frac{i_{\gamma(t)}}{u}}\omega,
\end{equation}
which is compatibles with $k((u))\widehat{\otimes}_k \O_M$-module structures, Gau\ss--Manin connections 
and multiplications of $u$ on both sides.
\end{prop}
This idea together with the equation~\eqref{33} below is the essential part of the construction of primitive forms, 
which is implicit in M.~Saito's paper~\cite{MSaito}, later re-discovered by Barannikov~\cite{bar:1,bar:2} and 
recently used by C.~Li--S.~Li--K.~Saito~\cite{LLS} and M.~Saito~\cite{MSaito2}.
\subsection{Primitive forms}
Pull the $\O_M$-linear extension of the higher residue parings $K_{\f_A}$ back to $\H_{\ff_A}$ by 
the isomorphism $\J_{\ff_A}$ and denote it by $K_{\ff_A}$.
We can now introduce the notion of primitive forms.
\begin{defn}[cf. K.~Saito, Definition~3.1 in \cite{s:1}]\label{primitive form zeta}
Let $r\in k$. An element $\zeta\in\H_{\ff_A}^{(0)}$ is called a formal primitive form with the minimal exponent $r$ 
for the tuple $(\H_{\ff_A}^{(0)},\nabla^\gamma, K_{\ff_A})$ if it satisfies following five conditions$;$
\begin{enumerate}
\item $u\nabla^\gamma_{e}\zeta=\zeta$ and $\zeta$ induces an $\O_M$-isomorphism:
\begin{equation}\tag{P1}\label{P1}
\T_M[[u]]\cong \H_{\ff_A}^{(0)},\quad \sum_{p=0}^\infty\delta_p u^{p}\mapsto 
\sum_{p=0}^\infty(u\nabla^\gamma_{\delta_p}\zeta)u^p.
\end{equation}
\item We have
\begin{equation}\tag{P2}\label{P2}
K_{\ff_A}(u\nabla^\gamma_\delta\zeta,u\nabla^\gamma_{\delta'}\zeta)\in k\cdot u^{w},\quad\delta,\delta'\in\T_M.
\end{equation}
\item We have
\begin{equation}\tag{P3}\label{P3}
\nabla^\gamma_{u\frac{d}{du}+E}\zeta=r\zeta.
\end{equation}
\item There exists a connection $\ns$ on $\T_M$ such that 
\begin{equation}\tag{P4}\label{P4}
u\nabla^\gamma_X\nabla^\gamma_{Y}\zeta=\nabla^\gamma_{X\circ Y}\zeta
+u\nabla^\gamma_{\ns_X Y}\zeta,\quad X,Y\in\T_M.
\end{equation}
\item There exists an $\O_M$-endomorphism $N:\T_M\longrightarrow \T_M$ such that 
\begin{equation}\tag{P5}\label{P5}
u\nabla^\gamma_\frac{d}{du}(u\nabla^\gamma_X\zeta)=-\nabla^\gamma_{E\circ X}\zeta+u\nabla^\gamma_{NX}\zeta,\quad 
X\in\T_M.
\end{equation}
\end{enumerate}
\end{defn}
We obtain a formal primitive form by applying a famous construction of primitive forms developed by M.~Saito and Barannikov.
\begin{thm}
There exists a formal primitive form $\zeta$ with the minimal exponent zero for the tuple $(\H_{\ff_A}^{(0)},\nabla^\gamma, K_{\ff_A})$.
\end{thm}
\begin{proof}
If a graded $k$-submodule $S$ of $\H_{\f_A}$ satisfies
\[
\H_{\f_A}=\H_{\f_A}^{(0)}\oplus S,\ 
u^{-1}S\subset S,\ 
\nabla_{u\frac{d}{du}}S\subset S,\ 
K_{\f_A}(S,S)\subset k[u^{-1}]u^{w-2},
\]
then the unique element $\zeta\in \H_{\ff_A}^{(0)}$ such that
\begin{equation}\label{33}
\J_{\ff_A}(\zeta)=\J_{\ff_A}({\H}_{\ff_A}^{(0)})\cap \left(\zeta_1+S\widehat{\otimes}_k\O_M\right),
\end{equation} 
becomes a formal primitive form, whose minimal exponent is zero due to the second property of Proposition~\ref{prop:Hodge numbers}.
Proposition~\ref{prop:marking} shows that such an $S$ exists.
\end{proof}
\subsection{Frobenius structures}
By a standard procedure, the existence of a formal primitive form implies the existence of a formal Frobenius structure 
(cf. Theorem~7.5 in \cite{st:1}, which can be generalized to graded cases verbatim).
\begin{thm}
Let $\zeta$ be a formal primitive form with the minimal exponent zero for the tuple $(\H_{\ff_A}^{(0)},\nabla, K_{\ff_A})$.
Define an $\O_M$-bilinear form $\eta_{\zeta^{\otimes 2}}:\T_M\otimes_{\O_M}\T_M\longrightarrow \O_M$ by
\begin{equation}
\eta_{\zeta^{\otimes 2}}(X,Y):=(-1)^{w\cdot \overline{Y}}K_{\ff_A}(u\nabla_X\zeta,u\nabla_Y\zeta).
\end{equation}
Then the tuple $(\circ, \eta_{\zeta^{\otimes 2}},e, E)$ gives a formal Frobenius structure on $\T_M$ such that 
$Lie_E(\circ)=\circ$ and $Lie_E(\eta_{\zeta^{\otimes 2}})=(2-w)\eta_{\zeta^{\otimes 2}}$. 
\end{thm}

\end{document}